\documentclass[12pt, a4paper]{amsart}
 \usepackage{amsaddr}
\usepackage[onehalfspacing]{setspace}
 \usepackage{stmaryrd}
\usepackage{fullpage}
\usepackage{amsmath}
\usepackage{epigraph}
\usepackage{amsmath, amscd}
\usepackage{IEEEtrantools}
\usepackage{amssymb, tikz}
\usepackage{mathabx}
\usepackage{mathtools}
\usepackage{comment}

\calclayout

\newtheorem{thm}{Theorem}[section]
\newtheorem{prop}[thm]{Proposition}
\newtheorem{lem}[thm]{Lemma}
\newtheorem{cor}[thm]{Corollary}

\theoremstyle{definition}
\newtheorem{defi}[thm]{Definition}

\theoremstyle{definition}

\theoremstyle{remark}

\newtheorem{remark}[thm]{Remark}

\numberwithin{equation}{section}

\def\G{\Gamma}

\newcommand{\R}{\mathbf{R}}  
\newcommand{\Z}{\mathbf{Z}}

\newcommand{\N}{\mathbf{N}}
\newcommand{\C}{\mathbf{C}}

\newcommand{\SL}{\mathrm{SL}_2 (\mathbf{Z})}

\newcommand{\floor}[1]{\left \lfloor #1 \right \rfloor}

\newcommand{\uh}{\mathbf{H}}

\newcommand{\sgn}{\mathrm{sign}}

\newcommand\SmallMatrix[4]{{\tiny\arraycolsep=0.3\arraycolsep\ensuremath{\begin{pmatrix}#1 & #2 \\ #3 & #4\end{pmatrix}}}}

\def\XXint#1#2#3{{\setbox0=\hbox{$#1{#2#3}{\displaystyle\displaystyle\int}$ }
\vcenter{\hbox{$#2#3$ }}\kern-.6\wd0}}

\usepackage{blindtext}
\usepackage{soul}

\usepackage{etoolbox}

\makeatletter
\patchcmd{\@setauthors}{\MakeUppercase\@author}{}{}{}
\makeatother

\begin{document}
\title{Uniform Distribution of Hardy Sums}
\author{Alessandro L\"ageler}
\maketitle

\begin{abstract}
    We employ the spectral theory of Eisenstein series to prove that the Hardy sums, integer-valued analogs of the classical Dedekind sums, are uniformly distributed in $\Z / m \Z$ for any integer $m > 1$.
\end{abstract}

\section{Introduction}

The classical Dedekind sums \begin{equation} \label{dededefi}
    s(d, c) = \frac{1}{4c} \sum_{k = 1}^{\vert c \vert - 1} \cot \frac{\pi k}{c} \cot \frac{\pi k d}{c}, \; (d, c) = 1,
\end{equation}
have been objects of interest ever since Dedekind first introduced them in his study \cite{dedekind} of the transformation behavior of the logarithm of $\eta(z) = e^\frac{\pi i z}{12} \prod_{n = 1}^\infty \left( 1 - e^{2 \pi i n z} \right)$. A lot is known about their distributional behavior: Hickerson  \cite{hickerson} showed that the set $\{ (d / c, \; s(d, c)) : (d, c) = 1, \; c > 0 \}$ is dense in $\R \times \R$, a result of Vardi \cite{vardi} asserts that their fractional parts are equidistributed on $\R / \Z$, and Myerson \cite{myerson} strenghtened Vardi's result to $(d / c, \; s(d, c))$ being equidistributed in $\R / \Z \times \R / \Z$.

A natural analog of the Dedekind sums in (\ref{dededefi}) are the so-called Hardy sums, which arise from the transformation behavior of the logarithm of $\theta$-functions. More precisely, let $\theta(z) = \sum_{n \in \Z} e^{\pi i n^2 z}$ and $\theta_4(z) = \sum_{n \in \Z} (-1)^n e^{\pi i n^2 z}$. Their logarithms transform as \begin{align}
\begin{split}\label{thetatrans}
    \log \theta (g.z) - \log \theta(z) &= \frac{1}{2} \log \left( \frac{cz + d}{i} \right) + \frac{\pi i }{4} S(d, c) \\
    &\mathrm{for} \; g = \SmallMatrix{*}{*}{c}{d} \in \SL, \; c > 0, \; c + d \; \mathrm{odd},
\end{split}
\end{align}
and \begin{align*} 
    \log \theta_4(g.z) - \log \theta_4(z) &= \frac{1}{2} \log\left(\frac{cz + d}{i} \right) - \frac{\pi i}{4} S_4(d, c) \\
    &\mathrm{for} \; g = \SmallMatrix{*}{*}{c}{d} \in \SL, \; c > 0, \; d \; \mathrm{odd},
\end{align*}
where $S(d, c)$ and $S_4(d, c)$ are integer-valued correction factors. Berndt \cite{berndt} proved that these correction factors have the following representation as finite sums:\begin{equation}
    S(d, c) = \sum_{k = 1}^{\vert c \vert - 1} (-1)^{k + 1 + \floor{\frac{dk}{c}}}, \; \; (d, c) = 1, \; c + d \; \mathrm{odd},
\end{equation}
and \begin{equation} \label{s4defi}
    S_4(d, c) = \sum_{k = 1}^{\vert c \vert - 1} (-1)^{\floor{\frac{dk}{c}}}, \; \; (d, c) = 1, \; d \; \mathrm{odd},
\end{equation}
which we refer to as "Hardy sums" (as they already appeared in early work of Hardy \cite{hardy}).

Equivalently to (\ref{thetatrans}), the Hardy sum $S(d, c)$ might be viewed in terms of the multiplier system of $\theta(z)^{8r} = e^{8r \log \theta (z)}$ for a rational number $0 < r < 1$ under the action of the group $$\G_\theta = \left\{ \SmallMatrix{a}{b}{c}{d} \in \SL : (d, c) = 1, \; c + d \; \mathrm{odd} \right\}.$$
Namely, the function $\theta(z)^{8r}$ transforms for $g = \SmallMatrix{*}{*}{c}{d}\in \G_\theta$ like $ \theta(g.z)^{8r} = \nu_r(g) (cz + d)^{4r} \theta(z)^{8r}$, where \begin{equation} \label{multsys}
    \nu_r(g) = \begin{cases} e^{2 \pi i r \left( S(d, c) - \sgn(c)\right)},  &c \neq 0, \\ e^{2 \pi i r (\sgn(d) - 1)}, &c = 0. \end{cases}
\end{equation}

Similarly as for the Dedekind sums $s(d, c)$, it is known that the sets $\{(d / c, \; S(d, c) ) : (d, c) = 1, \; c > 0, \; c + d \; \mathrm{odd} \}$ and $\{(d / c, \; S_4(d, c) ) : (d, c) = 1, \; c > 0, \; d \; \mathrm{odd} \}$ are dense in $\R \times \Z$; see \cite{meyer, laeg}. However, an analog of Vardi's result for $s(d, c)$ being equidistributed in $\R / \Z$ has been missing so far, a gap which we would like to fill with this note. 

To show that a real sequence $(a_n)_{n \geq 1}$ is equidistributed in $\R / \Z$ is equivalent to showing that for each $m \neq 0$ the Weyl sums \begin{equation} \label{weylcrit}
    \frac{1}{N} \sum_{n = 1}^N e(m a_n) \to 0, \; \mathrm{as} \; N \to + \infty,
\end{equation}
where we set $e(z) = e^{2 \pi i z}$ as usual.

As the Hardy sums are integer-valued, the question of equidistribution on $\R / \Z$ is trivial -- clearly, they are not equidistributed. Hence, instead of considering the Lebesgue-measurable set $\R / \Z$ we look at the distribution of $S(d, c)$ and $S_4(d, c)$ on the natural analog $\Z / m \Z$ for $m > 1$ with the counting measure. We get the following theorem. 

\begin{thm} \label{mainthm}
Let $m > 1$ be an integer. 
\begin{enumerate}
    \item The set $$\{ (d / c, \; S(d, c) \pmod m ): 1 \leq d < c, \; (d, c) = 1, \; c + d \; \mathrm{odd} \}$$ is uniformly distributed on $\R / \Z \times \Z / m \Z$ as $c \to + \infty$.
    \item The set $$\{ (d / c, \; S_4(d, c) \pmod m ): 1 \leq d < c, \; (d, c) = 1, \; d \; \mathrm{odd} \}$$ is uniformly distributed on $\R / \Z \times \Z / m \Z$ as $c \to + \infty$.
\end{enumerate}
\end{thm}

As an immediate corollary to Theorem \ref{mainthm} we get that the sets $\{ S(d, c) \pmod m : 1 \leq d < c, \; (d, c) = 1, \; c + d \; \mathrm{odd} \}$ and $\{ S_4(d, c) \pmod m : 1 \leq d < c, \; (d, c) = 1, \; d \; \mathrm{odd} \}$ are both uniformly distributed on $\Z / m \Z$ as $c \to + \infty$. 

It should be noted that the distribution in $\Z / m \Z$ of the sequence of the integer-valued Rademacher symbols $$\Psi \left( \SmallMatrix{a}{b}{c}{d} \right) = \frac{a + d}{c} - 12 \; s(d, c) - 3 \sgn(c(a + d)), \; \SmallMatrix{a}{b}{c}{d} \in \SL,$$
ordered by absolute value of the trace is known to be uniform \cite{ueki}. Moreover, other distributional results have been obtained for the Rademacher symbol $\Psi$ by Mozzochi \cite{mozzochi}, based on an idea of Sarnak (see also Von Essen's generalization to other multiplier systems \cite{vonessen}, which is relevant in our case).

In the next section, we will recall the generalization of equidistribution to compact abelian groups. The first statement of Theorem \ref{mainthm} is equivalent to showing that for all $n \in \Z$ and every rational number $0 < r < 1$: \begin{equation} \label{weylsums}
    \frac{1}{\# \Phi_\theta(N)} \sum_{c = 1}^N \sum_{\substack{d(c)^*, \\ c + d \; \mathrm{odd}}} e\left(- n \frac{d}{c} + r S(d, c)\right) \to 0, \; \mathrm{as} \; N \to + \infty,
\end{equation} 
where $\sum_{d(c)^*}$ means that we sum over $d \pmod c$ so that $(d, c) = 1$ and \begin{equation*} \label{phithetadefi}
    \Phi_\theta(N) = \{ d / c : 1 \leq d < c \leq N, \; (d, c) = 1, \; c + d \; \mathrm{odd} \}.
\end{equation*}
The second part of Theorem \ref{mainthm} is equivalent to the correspondent statement for $S_4(d, c)$ in (\ref{s4defi}). As the methods are virtually the same for $S(d, c)$ and $S_4(d, c)$, we will only focus on the Hardy sum $S(d, c)$ in this note. 

The idea for the proof is straightforward: The Dirichlet series of the Weyl sums in (\ref{weylsums}) appear in the Fourier expansion of certain Eisenstein series of weight $4r$ with multiplier system $\nu_r$ given in (\ref{multsys}). The spectral theory of Eisenstein series gives the analytic continuation of these Dirichlet series and information about its poles. Analytically, this encodes information about the growth rate of the Weyl sums (\ref{weylsums}). 

In section \ref{equidist}, we will first recall the notion of equidistribution on $\Z / m \Z$, give the asymptotic growth rate of $\# \Phi_\theta(N)$, and then prove Theorem \ref{mainthm} elementarily for $m = 2$. In section \ref{eisstsec}, we recall the some facts from the spectral theory of Eisenstein series, which we will use to give the estimates of the Weyl sums for the case $m > 2$. In section \ref{unifsec} we apply Perron's formula to prove that the exponential sums in (\ref{weylsums}) grow slower than $\# \Phi_\theta(N)$, hence proving Theorem \ref{mainthm}.

\section{Weyl Sums} \label{equidist}

\subsection{Equidistribution on Compact Abelian Groups}
The theory of equidistribution generalizes to the more general setting of compact abelian groups.

\begin{defi} \label{defgroupequi}
Let $G$ be a compact abelian group. A sequence $(g_n)_{n \in \N}$ is called equidistributed in $G$ if the measure $\frac{1}{N} \sum_{n = 1}^N \delta_{g_n}$ converges weakly to the Haar measure $\mu$ in $G$ as $N \to + \infty$, i.e. if for every continuous function $f : G \to \C$ we have $$\frac{1}{N} \sum_{n = 1}^N f(g_n) \to \int_G f(g) d \mu (g), \; \mathrm{as} \; N \to + \infty.$$
\end{defi}

Equidistribution modulo $1$ is a special case of Definition \ref{defgroupequi} for the compact abelian group $\R / \Z$ (the Haar measure on $\R / \Z$ is the Lebesgue measure). In our discussion of the integer-valued Hardy sums, we will then no longer consider $\R / \Z$, but the compact abelian group $\Z / m \Z$ for an integer $m > 1$ instead. The Haar measure on $\Z / m \Z$ is the counting measure. 

\begin{thm}[Eckmann \cite{eckmann}] \label{weyleckmann}
Let $(g_n)_{n \in \N}$ be a sequence in a compact abelian group $G$. The sequence $(g_n)_{n \in \N}$ is equidistributed in $G$ if and only if $\frac{1}{N} \sum_{n = 1}^N \chi(g_n) \to 0$ as $N \to + \infty$ for all non-trivial characters $\chi$ on $G$. 
\end{thm}

The Weyl criterion (\ref{weylcrit}) is a special case of Theorem \ref{weyleckmann}, as the non-trivial characters on $\R / \Z$ are given by $\{ e(nx) : n \neq 0 \}$. In the case $\Z / m \Z$ for $m > 1$, the non-trivial characters are given by $n \mapsto e \left( \frac{j}{m} n \right)$ for $j = 1, ..., m - 1$; see \cite{niven, uchiyama}. 

To show equidistribution of the Hardy sums $S(d, c)$ on, say, $\Z / 2 \Z$, we must show that \begin{equation}
    \frac{1}{\# \Phi_\theta(N)} \sum_{c = 1}^N \sum_{\substack{d(c)^*, \\ c + d \; \mathrm{odd}}} e \left( \frac{1}{2} S(d, c) \right) \to 0, \; \mathrm{as} \; N \to + \infty,
\end{equation}
and, more generally, that $$\frac{1}{\# \Phi_\theta(N)} \sum_{c = 1}^N \sum_{\substack{d(c)^*, \\ c + d \; \mathrm{odd}}} e \left( r S(d, c) \right) \to 0, \; \mathrm{as} \; N \to + \infty$$
for each rational $0 < r < 1$ to show equidistribution of $S(d, c)$ on $\Z / m \Z$ for all $m > 1$. 

Finally, for the coupled equidistribution in the first part of Theorem \ref{mainthm} we need to prove (\ref{weylsums}).

\subsection{Estimating $\Phi_\theta(N)$} Naturally, to prove equidistribution we need to understand the growth rate of the function \begin{equation}
    \# \Phi_\theta(N) = \# \{ d / c : 1 \leq d < c \leq N, \; (d, c) = 1, \; c + d \; \mathrm{odd} \}.
\end{equation}

We may also write the function $\# \Phi_\theta(N)$ as a sum of special Euler-totient functions $$\varphi_\theta(c) = \sum_{\substack{d(c)^*, \\ c + d \; \mathrm{odd}}} 1 = \begin{cases} \varphi(n), &\mbox{if } c \; \mathrm{even}, \\ \frac{1}{2} \varphi(n), &\mbox{if } c \; \mathrm{odd}. \end{cases}$$
Note that $\varphi_\theta(c) = \sum_{\substack{d(c)^*, \\ c + d \; \mathrm{odd}}} e \left(- n d / c \right)$ for $n = 0$. In this notation, we have $\# \Phi_\theta(N) = \sum_{c = 1}^N \varphi_\theta(c)$. 

\begin{lem} \label{asymplem}
We have $\Phi_\theta(N) \asymp N^2$ as $N \to + \infty$. 
\end{lem}

\begin{proof}
First, we notice that \begin{equation} \label{equationinproofsplitsum}
    \Phi_\theta(N) = \sum_{c = 1}^N \varphi_\theta(c) = \sum_{c = 1}^{\floor{N / 2}} \varphi(2c) + \frac{1}{2} \sum_{c = 1}^{\floor{N / 2}} \varphi(2c - 1).
\end{equation}
Let us denote the first sum on the right hand side of (\ref{equationinproofsplitsum}) by $\Lambda_1(N)$ and the second sum by $\Lambda_2(N)$. 

We get for $\Lambda_1(N)$:
\begin{align*}
    \Lambda_1(N) &= \sum_{k = 1}^{\floor{N / 2}} \varphi(2k) \\
    &= \sum_{k = 1}^{\floor{N / 2}} \sum_{dd' = 2k} d' \mu(d) \\
    &= \sum_{\substack{d = 1, \\ d \; \mathrm{even}}}^N \mu(d) \sum_{d' = 1}^{\floor{N / d}} d' + \sum_{\substack{d = 1, \\ d \; \mathrm{odd}}}^{\floor{N / 2}} \mu(d) \sum_{\substack{d' = 1, \\ d' \; \mathrm{even}}}^{\floor{N / d}} d' \\
    &= - \sum_{\substack{d = 1, \\ d \; \mathrm{odd}}}^{\floor{N / 2}} \mu(d) \sum_{d' = 1}^{\floor{N / 2d}} d' + 2 \sum_{\substack{d = 1, \\ d \; \mathrm{odd}}}^{\floor{N / 2}} \mu(d) \sum_{d' = 1}^{\floor{N / 2d}} d' \\
    &= \sum_{\substack{d = 1, \\ d \; \mathrm{odd}}}^{\floor{N / 2}} \mu(d) \sum_{d' = 1}^{\floor{N / 2d}} d' \\
    &= \sum_{\substack{d = 1, \\ d \; \mathrm{odd}}}^{\floor{N / 2}} \mu(d) \frac{1}{2} \floor{\frac{N}{2d}} \left( \floor{\frac{N}{2d}} + 1 \right)\\
    &= \frac{N^2}{8} \sum_{\substack{d = 1, \\ d \; \mathrm{odd}}}^\infty \frac{\mu(d)}{d^2} + o(N^2).
\end{align*}

Similarly, one can show that $\Lambda_2(N) = \frac{N^2}{8} \sum_{\substack{d = 1, \\ d \; \mathrm{odd}}}^\infty \frac{\mu(d)}{d^2} + o(N^2)$.
\end{proof}

\begin{remark} \label{case2rem}
Since the asymptotic leading terms of $\Lambda_1(N)$ and $\Lambda_2(N)$ in the proof of Lemma \ref{asymplem} are equal, we have $\Lambda_1(N) - \Lambda_2(N) = o(N^2)$ as $N \to + \infty$. 
\end{remark}

Thus, by Lemma \ref{asymplem}, to prove Theorem \ref{mainthm}, we need to show that for all rational numbers $0 < r < 1$ and $n \neq 0$ the Weyl sums \begin{equation} \label{weylsumslittleo}
    \sum_{c = 1}^N \sum_{\substack{d(c)^*, \\ c + d \; \mathrm{odd}}} e(- n d / c + r S(d, c)) = o(N^2)
\end{equation} 
as $N \to + \infty$.

\subsection{The case $m = 2$} \label{casem2}

We start by reviewing the proof of the equidistribution of the sequence $\{ d / c : 1 \leq d < c, \; (d, c) = 1 \}$ in $\R / \Z$. The proof amounts to finding non-trivial cancellation in the Ramanujan sums $R_c(n) = \sum_{d(c)^*} e\left( n \frac{d}{c} \right)$. This follows from \begin{equation} \label{ramanujanbound}
    R_c(n) \leq n \; \mathrm{for \; all \;} n \neq 0,
\end{equation}
which can easily be seen from the von Sterneck formula $R_c(n) = \mu \left( \frac{c}{(c, n)} \right) \frac{\varphi(c)}{\varphi\left( \frac{c}{(c, n)} \right)}$, where $\mu$ is the M\"obius-function. Hence, \begin{equation} \label{ramanujanweyl}
    \left\vert \frac{1}{N^2} \sum_{c = 1}^N \sum_{d(c)^*} e \left( n \frac{d}{c} \right) \right\vert \leq \frac{1}{N^2} Nn \to 0, \; \mathrm{as} \; N \to + \infty,
\end{equation}
by (\ref{ramanujanbound}) and thus $\{ d / c : 1 \leq d < c, \; (d, c) = 1 \}$ is equidistributed on $\R / \Z$. 

Let us now consider the sequence $\{ d / c : 1 \leq d < c, \; (d, c) = 1, \; c + d \; \mathrm{odd} \}$ instead. This sequence is also equidistributed in $\R / \Z$. Consider the Weyl sums \begin{equation} \label{case2n}
    \sum_{c = 1}^N \sum_{\substack{d(c)^*, \\ c + d \; \mathrm{odd}}} e\left( n \frac{d}{c} \right) = \sum_{\substack{c \leq N, \\ c \; \mathrm{even}}} \sum_{d(c)^*} e\left(n \frac{d}{c} \right) + \sum_{\substack{c \leq N, \\ c \; \mathrm{odd}}} \sum_{\substack{d(c)^*, \\ d \; \mathrm{even}}} e\left(n \frac{d}{c} \right)
\end{equation} 
for $n \neq 0$. The first sum on the right hand side of (\ref{case2n}) is seen to be a sum of Ramanujan sums $R_c(n)$, which is $o(N^2)$ by (\ref{ramanujanweyl}). For the second sum, it is easy to show that $$\sum_{\substack{d(c)^*, \\ d \; \mathrm{even}}} e\left( - n \frac{d}{c} \right) = \sum_{\substack{0 < d < \floor{\frac{c}{2}}, \\ (d, c) = 1}} e\left( - 2n \frac{d}{c} \right) \leq \frac{\vert R_c(2n) \vert + 1}{\left\vert 1 + e \left( -2n \frac{\floor{c / 2}}{\floor{c / 2} + 1} \right) \right\vert} \ll n$$
uniformly in $c$. This proves that the Weyl sum $\sum_{c = 1}^N \sum_{\substack{d(c)^*, \\ c + d \; \mathrm{odd}}} e\left(n \frac{d}{c} \right) = o(N^2)$ as well and thus proves equidistribution of $\{ d / c : 1 \leq d < c, \; (d, c) = 1, \; c + d \; \mathrm{odd} \}$ on $\R / \Z$.

We now give an elementary proof of the equidistribution of $(d / c, S(d, c) \pmod 2)$ on $\R / \Z \times \Z / 2 \Z$. To this end, we need to bound the Weyl sum $$\sum_{c = 1}^N \sum_{\substack{d(c)^*, \\ c + d \; \mathrm{odd}}} e \left( n d / c \right) e^{\pi i S(d, c)}.$$ As we showed in a previous paper \cite[Lem. 4.1]{laeg}, the Hardy sum $S(d, c)$ is odd when $c$ is even and $S(d, c)$ is even when $c$ is odd.\footnote{Similarly, for the value of $S_4(d, c)$ for $c > 0$, $(d, c) = 1$, and $d$ odd is an odd number if $c$ is even and is an even number if $c$ is odd \cite[Sec. 4.2]{laeg}.} Thus, we have $e^{\pi i S(d, c)} = (-1)^{c + 1}$ and so \begin{equation} \label{weylsum2}
    \sum_{c = 1}^N \sum_{\substack{d(c)^*, \\ c + d \; \mathrm{odd}}} e \left( n d / c \right) e^{\pi i S(d, c)} = \sum_{\substack{c \leq N, \\ c \; \mathrm{odd}}} \sum_{\substack{d(c)^*, \\ d \; \mathrm{even}}} e\left(n \frac{d}{c} \right) -  \sum_{\substack{c \leq N, \\ c \; \mathrm{even}}} \sum_{d(c)^*} e\left(n \frac{d}{c} \right).
\end{equation}
That (\ref{weylsum2}) is $o(N^2)$ follows immediately from the bounds we used to bound (\ref{case2n}) in the case $n \neq 0$. 

For $n = 0$, we see that $$\sum_{c = 1}^N \sum_{\substack{d(c)^*, \\ c + d \; \mathrm{odd}}} e^{\pi i S(d, c)} = \sum_{c = 1}^N (-1)^{c + 1} \varphi_\theta(c) = \Lambda_2(N) - \Lambda_1(N)$$
in the notation of the proof of Lemma \ref{asymplem}. By Remark \ref{case2rem} the Weyl sum is thus $o(N^2)$ and this proves the first part of (1) in Theorem \ref{mainthm} for $m = 2$.

Albeit it is probably possible to give an elementary proof of Theorem \ref{mainthm} as well for $m > 2$, we can use the spectral theory of automorphic forms instead. The information of the growth rate of the Weyl sums is encoded in the analytic behavior of certain Eisenstein series. The spectral method could also be used to show the case $m = 2$; but since Eisenstein series of weight $2$ are more delicate, we chose to give an elementary proof in that case.

\section{Eisenstein Series with Multiplier Systems} \label{eisstsec}

Let $0 < r < 1$ denote a rational number. Our main tool to prove the uniform distribution of Hardy sums will be the theory of Eisenstein series with a multiplier system. By a multiplier system of weight $r$ on a subgroup $\G$ of $\SL$ we mean a function $\nu : \G \to \C$ such that the cocycle $j(g, z) = cz + d$ with $g = \SmallMatrix{*}{*}{c}{d} \in \G$ satisfies $\nu(gh) j(gh, z)^r = \nu(g) \nu (h) j(g, h.z)^r j(h, z)^r$ for all $g, h \in \G$. 
 
An example of a multiplier system $\nu_r$ for $\G_\theta$ was already given in (\ref{multsys}), which is the multiplier system of the weight $4r$ modular form $\theta(z)^{8r} = e^{8r \log \theta (z)}$ under $\G_\theta$. It is not difficult to see that $\nu_{1 / 2}$ is a multiplicative homomorphism. 

For the group $\G_\theta$, let $\G_\infty < \G_\theta$ denote the subgroup of its parabolic elements, i.e. the subgroup generated by $\pm \SmallMatrix{1}{2}{0}{1}$. The multiplier system $\nu_r$ is singular with respect to the cusp $i \infty$, i.e. $\nu_r \left( \SmallMatrix{1}{2}{0}{1} \right) = 1$ (see also \cite[Prop. 3.4]{laeg}). We may hence view $\nu_r$ as a function on the cosets $\G_\infty \setminus \G$.

For a multiplier system $\nu$, consider the following Dirichlet series: \begin{equation}
    Z_\nu(n, s) = \sum_{c > 0} \frac{1}{c^{2s}} \sum_{\substack{d(c)^*, \\ c + d \; \mathrm{odd}}}  e\left( - n \frac{d}{c}\right) \nu \left( \SmallMatrix{*}{*}{c}{d} \right).
\end{equation}
For our bounds on the Weyl sums (\ref{weylsums}), we will be interested in the special case $$Z_r(n, s) = Z_{\nu_r}(n, s) = e^{- 2 \pi i r} \sum_{c = 1}^\infty \frac{1}{c^{2s}} \sum_{\substack{d(c)^*, \\ c + d \; \mathrm{odd}}}  e\left( - n \frac{d}{c} + r S(d, c) \right).$$

For a multiplier system $\nu$ of weight $4r$, which is singular with respect to $i \infty$, define \begin{equation} \label{eisensteindefi} E_{4r} (z, s; \nu) = \sum_{g \in \G_\infty \setminus \G_\theta} \overline{\nu(g)} j(g, z)^{-4r} \mathrm{Im}(g.z)^{s - 2r}, \; \mathrm{Re}(s) > 1,
\end{equation}
which is the non-holomorphic Eisenstein series of weight $4r$ and multiplier system $\nu$. The series is absolutely convergent for $\mathrm{Re}(s) > 1$ and uniformly convergent on compact subsets of $\uh$.


\begin{prop} \label{eisfourier}
The Fourier expansion of the non-holomorphic Eisenstein series is given by $$E_{4r}(z, s; \nu) = y^{s - 2r} + \varphi_{4r, \nu}(s) y^{1 - s - 2r} + \sum_{n \neq 0}  y^{-2r} \varphi_{4r, \nu}(n, s) W_{\sgn(n) 2r, s - 1 / 2}(2 \pi \vert n \vert y) e^{\pi i n x},$$
where \begin{align*}
    \varphi_{4r, \nu}(s) &= \pi \frac{2^{2 - 2s} \Gamma (2s - 1) e^{2 \pi i r}}{\Gamma (s - 2r) \Gamma(s + 2r)} Z_\nu(0, s), \\
    \varphi_{4r, \nu}(n, s) &= \frac{\pi^s \vert n \vert^{s - 1} e^{2 \pi i r}}{\Gamma (s + \sgn(n) 2r)} Z_\nu(n, s),
\end{align*}
and $W_{\kappa, \mu}(z)$ denotes the Whittaker $W$-function.
\end{prop}

\begin{proof}
See, for instance, Fay's paper \cite[Thm. 3.4]{fay}; note that the Eisenstein series considered in Fay's paper is $y^{- 4r} E_{- 4r}(z, s; \nu)$ in our notation.
\end{proof}

Proposition \ref{eisfourier} is the special case of the Fourier expansion around the cusp $\mathfrak{a} = i \infty$. More generally, for any cusp $\mathfrak{a}$ of $\G_\theta$, let $\sigma_\mathfrak{a} \in \mathrm{SL}_2(\R)$ be its scaling matrix such that $\sigma_\mathfrak{a}.i \infty = \mathfrak{a}$. The Eisenstein series $E_{4r}(z, s; \nu)$ has a Fourier expansion around $\mathfrak{a}$ given by \begin{align}
    \begin{split} \label{eisfouriercusp}
    j(\sigma_\mathfrak{a}, z)^{4r} E_{4r}(\sigma_\mathfrak{a}^{-1}.z, s; \nu) &= \delta_{\mathfrak{a}, \infty} \mathrm{Im}(\sigma_{\mathfrak{a}}^{-1}.z)^{s - 2r} + \varphi_{4r, \mathfrak{a}, \nu} (s) \mathrm{Im}(\sigma_{\mathfrak{a}}^{-1}.z)^{1 - s - 2r} \\
    &+ \sum_{n \neq 0} y^{-2r} \varphi_{4r, \mathfrak{a}, \nu}(n, s) W_{\sgn(n) 2r, s - \frac{1}{2}}(4 \pi \vert n \vert y) e^{\pi i n x}
    \end{split}
\end{align}
for certain functions $\varphi_{4r, \mathfrak{a}, \nu}(s)$ and $\varphi_{4r, \mathfrak{a}, \nu}(n, s)$. 

It is well-known that the Eisenstein series $E_{4r}(z, s; \nu)$ has an analytic continuation in $s$ to the whole complex plane. From the Fourier expansion in Proposition \ref{eisfourier}, the poles of $E_{4r}(z, s; \nu)$ correspond to poles of the analytically continued Dirichlet series $Z_\nu(n, s)$ or -- e.g. in the case $s = \frac{1}{2}$ -- to poles of the $\G$-factors. For a pole $s_0$ of $E_{4r}(z, s; \nu)$, we shall be interested in giving an upper bound for $\mathrm{Re}(s_0)$.

\begin{prop} \label{sqintres}
Suppose the Eisenstein series $E_{4r}(z, s; \nu)$ has a pole at $s_0 \in \C$ with $\mathrm{Re}(s_0) > \frac{1}{2}$. Then its residue $\mathrm{Res}_{s = s_0} E_{4r}(z, s; \nu)$ is a non-zero square-integrable eigenfunction of $\Delta_{4r}$ with eigenvalue $(s_0 - 2r) (1 - s_0 - 2r)$. 
\end{prop}

\begin{proof}
This can be seen from looking at the Fourier expansion of the residues around each cusp $\mathfrak{a}$ of $\G_\theta$ in (\ref{eisfouriercusp}).
\end{proof}

\begin{remark}
For $r = \frac{1}{8}$ the Eisenstein series $E_{1 / 2}(z, s; \nu_r)$ has a pole at $s = \frac{3}{4}$ and its residue is equal to $\theta(z)$. 
\end{remark}

\begin{prop} \label{corlocpoles}
The Eisenstein series $E_{4r}(z, s; \nu)$ is holomorphic on the half-plane $\mathrm{Re}(s) > \max \{1 - 2r, 2r\}$.
\end{prop}

\begin{proof}
Suppose $E_{4r}(z, s; \nu)$ has a pole at $s_0 \in \C$. By Proposition \ref{sqintres}, the residue of the Eisenstein series at $s_0$ is a non-zero square-integrable function $\varphi(z)$ with eigenvalue $(s_0 - 2r) (1 - s_0 - 2r)$ under $\Delta_{4r}$, i.e. a Maass form. Each Maass form has non-negative eigenvalue. An elementary analysis of the polynomial $(s - 2r) (1 - s - 2r) = - s^2 + s + 4r^2 \geq 0$ gives that for $\mathrm{Re}(s) > \frac{1}{2}$ the point $s$ needs to be real and that its value will be negative if $\mathrm{Re}(s) > \max \{1 - 2r, 2r \}$. Hence, we must have $\mathrm{Re}(s_0) \leq \max\{1 - 2r, 2r\}$. 
\end{proof}

Generally, the Eisenstein series $E_{4r}(z, s; \nu)$ has a pole at $s = \max \{1 - 2r, 2r \}$ for $0 < r < \frac{1}{2}$. The Eisenstein series $E_2(z, s; \nu)$ does not have a pole at $s = 1$ and is a harmonic function at that point. 

Moreover, Proposition \ref{corlocpoles} does not give any new information about the poles of the Eisenstein series $E_{4r}(z, s; \nu)$ in the strip $0 < \mathrm{Re}(s) < 1$ for $\frac{1}{2} < r < 1$. However, using lowering operators, we can easily get information about the poles of $E_{4r}(z, s; \nu)$ in the strip through Proposition \ref{corlocpoles}.

\begin{cor} \label{corlocpoles2}
For $\frac{1}{2} < r < 1$, the Eisenstein series $E_{4r}(z, s; \nu)$ is holomorphic for $\mathrm{Re}(s) > 2r - 1$.
\end{cor}

\begin{proof}
This follows from Proposition \ref{corlocpoles}, using $$2 i y^2 \partial_{\overline{z}} E_{4r}(z, s; \nu) = (s - 2r) E_{4r - 2}(z, s; \nu)$$ and looking at the Laurent-expansion around a pole of $E_{4r}(z, s; \nu)$.
\end{proof}

From the Fourier expansion of $E_{4r}(z, s; \nu)$ we see that if the Eisenstein series is holomorphic at some point $s$ with $\mathrm{Re}(s) > \frac{1}{2}$, then the Dirichlet series $Z_\nu(n, s)$ are holomorphic at $s$ as well. We rephrase Proposition \ref{corlocpoles} and Corollary \ref{corlocpoles2} in a separate corollary in the special case of the Dirichlet series $Z_{r}(n, s)$, which we will apply in the next section.

\begin{cor} \label{corthatweuse}
The functions $Z_r(n, s)$ are holomorphic for $\mathrm{Re}(s) > \max \{1 - 2r, 2r \}$ for $0 < r < \frac{1}{2}$ and for $\mathrm{Re}(s) > 2r - 1$ for $\frac{1}{2} < r < 1$.
\end{cor}


The Eisenstein series $E_{4r}(z, s; \nu)$ as such is not square-integrable with respect to the Petersson inner product $$\langle f, g \rangle = \int_{\G_\theta \setminus \uh} f(z) \overline{g(z)} y^{4r} \frac{dxdy}{y^2},$$
the problematic part of $E_{4r}(z, s; \nu)$ being the zeroth Fourier coefficient in (\ref{eisfouriercusp}). To make it square-integrable, we consider the truncated Eisenstein series instead.

Let $Y > 0$. For each cusp $\mathfrak{a}$, let $F_\mathfrak{a}(Y) = \sigma_\mathfrak{a}.F_Y$ be the cuspidal zone around $\mathfrak{a}$, where $\sigma_\mathfrak{a} \in \mathrm{SL}_2(\R)$ is such that $\sigma_\mathfrak{a}.i \infty = \mathfrak{a}$ and $F_Y = \{ z \in \uh : - 1 < \mathrm{Re}(s) < 1, \mathrm{Im}(z) > \max \{1, Y \} \}$. We define the truncated Eisenstein series by \begin{equation*}
    E_{4r}^Y(z, s; \nu) = \begin{cases} E_{4r}(z, s; \nu) - \delta_{\mathfrak{a}, \infty} \mathrm{Im}(\sigma_{\mathfrak{a}}^{-1}.z)^{s - 2r} - \varphi_{4r, \mathfrak{a}, \nu} (s) \mathrm{Im}(\sigma_{\mathfrak{a}}^{-1}.z)^{1 - s - 2r}, &z \in F_\mathfrak{a}(Y), \\ E_{4r}(z, s; \nu), &\mathrm{otherwise}.  \end{cases}
\end{equation*}

The truncated Eisenstein series $E_{4r}^Y(z, s; \nu)$ is now a square-integrable function on $\G_\theta \setminus \uh$ of weight $4r$. Moreover, it satisfies the Maass-Selberg relations.

\begin{thm} \label{maassselberg}
Let $s_1$ and $s_2$ with $s_1 \neq s_2$ be regular points of $E_{4r}(z, s; \nu)$ and let $E_{4r}^Y(z, s; \nu)$ be the truncated Eisenstein series. We have \begin{align*}
    \langle E_{4r}^Y(z, s_1; \nu), E_{4r}^Y(z, s_2; \nu) \rangle &= (s_1 - s_2)^{-1} \varphi_{4r, \nu}(s_1) + (s_2 - s_1)^{-1} \varphi_{4r, \nu}(s_2) \\
    &+ (s_1 + s_2 - 1)^{-1} Y^{s_1 + s_2 - 1} - (s_1 + s_2 - 1)^{-1} \varphi_{4r, \nu}(s_1) \overline{\varphi_{4r, \nu}(s_2)}.
\end{align*}
\end{thm}

\begin{proof}
This is an application of Green's theorem; see, e.g., Iwaniec's book \cite[Prop. 6.8]{iwaniec}.
\end{proof}

The Maass-Selberg relations in the case $\nu = \nu_r$ give a growth estimate for $Z_r(n, \sigma + it)$ with $\frac{1}{2} < \sigma < \frac{3}{2}$ as $t \to \pm \infty$, which we will use for giving a growth estimate for the Weyl sums (\ref{weylsums}). 

\begin{lem} \label{growthlem}
Let $\frac{1}{2} < \sigma < \frac{3}{2}$ and $n \in \Z$. We have \begin{equation*}
    Z_r(n, \sigma + it) \ll_{\sigma, \; r, \; n} \sqrt{\vert t \vert}
\end{equation*}
as $\vert t \vert \to + \infty$. 
\end{lem}

\begin{proof}
Write $s = \sigma + it$ with $\sigma \in \R$ and $t > 0$. From Theorem \ref{maassselberg} for $\nu = \nu_r$ with $s_1 = s$ and $s_2 = \overline{s}$, we get \begin{align*}
    &\int_{\G_\theta \setminus \uh} \vert E_{4r}^Y(z, s; \nu_r) \vert^2 d \mu(z) \\
    &= \frac{1}{2 \sigma - 1} \left( Y^{2 \sigma - 1} + \vert \varphi_{4r, \nu_r}(s) \vert^2 Y^{1 - 2 \sigma} \right) + \frac{1}{it} \left(\overline{\varphi_{4r, \nu_r}(s)} Y^{it} + \varphi_{4r, \nu_r}(s) Y^{- it} \right),
\end{align*} from which it follows immediately that $\varphi_{4r, \nu_r}(s)$ is bounded as $\vert t \vert \to + \infty$. By Stirling's approximation, we have $$\frac{\Gamma(2s - 1)}{\Gamma(s + 2r) \Gamma(s - 2r)} \sim \frac{\vert t \vert^{- 1 / 2}}{\sqrt{2 \pi}}, \; \mathrm{as} \; t \to \pm \infty,$$
hence the claim for $n = 0$ follows by Proposition \ref{eisfourier}. 

To prove the assertion for $Z_r(n, s)$ with $n \neq 0$, we again use the Maass-Selberg relations to estimate for $Y > 1$: \begin{align*}
    &\int_{\G_\theta \setminus \uh} \vert E_{4r}^Y(z, s; \nu_r) \vert^2 y^{4r} \frac{dxdy}{y^2} \\
    &\geq \int_{Y}^\infty \int_{-1}^1 \vert E_{4r}^Y(z, s; \nu_r) \vert^2 y^{4r} \frac{dxdy}{y^2} \\
    &= \sum_{m \neq 0} \left\vert \frac{\pi^s \vert m \vert^{s - 1}}{\Gamma(s + \sgn(m) 2r)} \right\vert^2 \vert Z_r(m, s) \vert^2 \int_Y^\infty \vert W_{\sgn(m) 2r, s - \frac{1}{2}}(4 \pi \vert m \vert y) \vert^2 \frac{dy}{y^2} \\
    &\geq \left\vert \frac{\pi^s \vert n \vert^{s - 1}}{\Gamma(s + \sgn(n) 2r)} \right\vert^2 \vert Z_r(n, s) \vert^2 \int_Y^{Y + 1} \vert W_{\sgn(n) 2r, s - \frac{1}{2}}(4 \pi \vert n \vert y) \vert^2 \frac{dy}{y^2}.
\end{align*}
Using the asymptotic formula $$W_{\sgn(n) 2r, s - \frac{1}{2}}(4 \pi \vert n \vert y) \sim \pi^{- 1 / 2} \Gamma(\sgn(n) 2r + s - 1 / 2) (\pi \vert n \vert y)^{1 - s}$$ as $\vert t \vert \to + \infty$ (which converges uniformly for $y$ in a compact set), we see that \begin{equation} \label{eqinlem}
    \int_{\G_\theta \setminus \uh} \vert E_{4r}^Y(z, s; \nu_r) \vert^2 y^{4r} \frac{dxdy}{y^2} \gg \left\vert \frac{\Gamma( s + \sgn(n) 2r - 1 / 2)}{\Gamma(s + \sgn(n) 2r)} \right\vert^2 \vert Z_r(n, s) \vert^2 \int_Y^{Y + 1} \frac{1}{y^{2 \sigma}} dy.
\end{equation}
The left hand side of (\ref{eqinlem}) is bounded by what we showed in the first part of the claim. Again, by Stirling's approximation, the claim follows for $n \neq 0$. 
\end{proof}

\section{Uniform Distribution of Hardy Sums} \label{unifsec}

In this section, we will use the results from section \ref{eisstsec} to prove Theorem \ref{mainthm}.

Let $0 < r < 1$ be a rational number and $n \in \Z$. Since the case $r = \frac{1}{2}$ was already proved in section \ref{casem2}, we may assume that $r \neq \frac{1}{2}$.

To work out the explicit growth rate of the Weyl sums, we will use Perron's formula, i.e. that for any $\alpha > 1$, we have \begin{equation}
    \sum_{c = 1}^N \sum_{\substack{d(c)^*, \\ c + d \; \mathrm{odd}}} e(n d / c + r S(d, c)) = \frac{1}{2 \pi i} \int_{\alpha - i \infty}^{\alpha + i \infty} Z_r(n, s) \frac{N^{2s}}{2s}ds.
\end{equation}

Let $\varepsilon > 0$, $T > 0$, $\alpha = 1 + \varepsilon$, and $s_1, ..., s_l \in \C$ be the location of the poles of $Z_r(n, s)$. Using Cauchy's theorem, we obtain \begin{align*}
    \frac{1}{2 \pi i} \int_{\alpha - i T}^{\alpha + i T} Z_r(n, s) \frac{N^{2s}}{2s}ds &= \int_{\frac{1}{2} + \varepsilon - i T}^{\frac{1}{2} + \varepsilon + i T} Z_r(n, s) \frac{N^{2s}}{2s} ds + \int_{\frac{1}{2} + \varepsilon + i T}^{\alpha + i T} Z_r(n, s) \frac{N^{2s}}{2s} ds \\
    &- \int_{\frac{1}{2} + \varepsilon - i T}^{\alpha - i T} Z_r(n, s) \frac{N^{2s}}{2s} ds + \sum_{j = 1}^n \mathrm{Res}_{s = s_j} Z_r(n, s) \frac{N^{2s_j}}{2s_j}.
\end{align*}

The integrals along the horizontal lines satisfy \begin{align*}
    \left\vert \int_{\alpha - i T}^{\frac{1}{2} + \varepsilon - iT} Z_r(n, s) \frac{N^{2s}}{2s}ds \right\vert \leq \int_{\alpha - i T}^{\frac{1}{2} + \varepsilon - iT} \left\vert Z_r(n, s) \frac{N^{2s}}{2s}\right\vert ds \leq \frac{N^{2\alpha}}{T} \int_{\alpha - i T}^{\frac{1}{2} + \varepsilon - iT} \left\vert Z_r(n, s) \right\vert ds \ll \frac{N^{2\alpha}}{T^\frac{1}{2}}
\end{align*}
as $T \to + \infty$ by Lemma \ref{growthlem}. The integral along the vertical line $\frac{1}{2} + \varepsilon + i \R$ can, again by Lemma \ref{growthlem}, be bounded by \begin{equation}
    \left\vert \int_{\frac{1}{2} + \varepsilon - i T}^{\frac{1}{2} + \varepsilon + i T} Z_r(n, s) \frac{N^{2s}}{2s} ds \right\vert \ll_{r, \; \varepsilon, \; n} N^{1 + \varepsilon} T^\frac{1}{2}.
\end{equation}

Hence, we get \begin{equation*}
    \sum_{c = 1}^N \sum_{\substack{d(c)^*, \\ c + d \; \mathrm{odd}}} e(n d / c + r S(d, c)) = \sum_{j = 1}^n \mathrm{Res}_{s = s_j} Z_r(n, s) \frac{N^{2s_j}}{2s_j} + O \left(N^{2 + \varepsilon} / T^{1 / 2} + N^{1 + \varepsilon} T^{1 / 2} \right)
\end{equation*}
as $T \to + \infty$. Theorem \ref{mainthm} then follows by Corollary \ref{corthatweuse} and choosing $T = N$.

\end{document}